\documentclass{amsart}

\usepackage{graphicx}
\usepackage{color}
\usepackage{enumerate}
\usepackage{amssymb}
\newtheorem{theorem}{Theorem}
\newtheorem{lemma}[theorem]{Lemma}

\theoremstyle{definition}

\theoremstyle{remark}
\newtheorem{remark}[theorem]{Remark}

\numberwithin{equation}{section}

\newcommand{\Hi}{\mathbb{H}}
\newcommand{\Sf}{\mathbb{S}}
\newcommand{\R}{\mathbb{R}}
\newcommand{\Lo}{\mathbb{L}}
\newcommand{\E}{\mathbb{E}}
\newcommand{\Q}{Q^{n}(\varepsilon)}

\begin{document}

\title{Einstein hypersurfaces of $\mathbb{S}^n \times \mathbb{R}$ and $\mathbb{H}^n \times \mathbb{R}$}

\author{Benedito Leandro}
\address{Benedito Leandro - Instituto de Matem\'atica e Estat\'istica, Universidade Federal de Goi\'as, Goi\^ania, 74001-970, Brazil}
 \email{bleandroneto@ufg.br}

\author{Romildo Pina}
\address{Romildo Pina - Instituto de Matem\'atica e Estat\'istica, Universidade Federal de Goi\'as, Goi\^ania, 74001-970, Brazil}
\email{romildo@ufg.br}

\author{Jo\~ao Paulo dos Santos}
\address{Jo\~ao Paulo dos Santos - Departamento de Matem\'atica, Universidade de
Bras\'ilia, 70910-900, Bras\'ilia-DF, Brazil}
\email{joaopsantos@unb.br}
\thanks{The third author was supported by FAPDF 0193.001346/2016}



\begin{abstract}

In this paper, we classify the Einstein hypersurfaces of $\mathbb{S}^n \times \mathbb{R}$ and $\mathbb{H}^n \times \mathbb{R}$. We use the characterization of the hypersurfaces of  $\mathbb{S}^n \times \mathbb{R}$ and $\mathbb{H}^n \times \mathbb{R}$ whose tangent
component of the unit vector field
spanning the factor $\mathbb{R}$ is a principal direction and the theory of isoparametric hypersurfaces of space forms to show that Einstein hypersurfaces of $\mathbb{S}^n \times \mathbb{R}$ and $\mathbb{H}^n \times \mathbb{R}$ must have constant sectional curvature. 

\end{abstract}

\maketitle

\section{Introduction} 


A Riemannian manifold $(M^n,g)$ is said to be Einstein if its Ricci tensor is proportional to the metric, i.e., if $Ric_M = \rho g$, for some constant $\rho \in \mathbb{R}$. Equivalently, $(M^n,g)$ is an Einstein manifold if it has constant Ricci curvature and,  according to Besse \cite{besse}, constant Ricci curvature could be considered as a good generalization of the concept of constant sectional curvature. Also, as pointed out in \cite{besse}, there are several results in the literature justifying that an Einstein metric is a good candidate  for a ``best'' metric on a given manifold. When $n=2$, the Einstein condition means constant Gaussian curvature whereas a simple calculation shows that, when $n=3$, a manifold $(M^n,g)$ is Einstein if and only if it has constant sectional curvature.

This paper aims to prove that an isometric immersion of an Einstein manifold $M^n$ as a hypersurface of the Riemannian products $\mathbb{S}^n \times \R$ and $\mathbb{H}^n \times \mathbb{R}$ only occur when $M^n$ has constant sectional curvature. More precisely, let us denote by $\Q$ the unit sphere $\mathbb{S}^n$, if $\varepsilon=1$, or the hyperbolic space $\mathbb{H}^n$, if $\varepsilon=-1$. With this notation, our main theorem is given as the following:

\begin{theorem}
Let $f: M^n \rightarrow Q^{n}(\varepsilon) \times \mathbb{R}$, $n > 3$, be an isometric immersion of an Einstein manifold. Then $M^n$ is a manifold with constant sectional curvature.  \label{thm-einstein}
\end{theorem}

Isometric immersions of Einstein manifolds into space forms were considered initially in codimension 1 by Thomas \cite{thomas}, followed by Fialkow \cite{fialkow} and the full  classification in this case was concluded by Ryan \cite{ryan}. Briefly, an Einstein hypersurface of a space form of curvature $\varepsilon$ must have constant sectional curvature, except for the case $\varepsilon=1$, where we can find a product of spheres as Einstein hypersurfaces. For arbitrary codimensions, Einstein submanifolds of space forms were considered recently under the hypothesis of having flat normal bundle. Onti \cite{onti} classified such submanifolds with parallel mean curvature, whereas Dajczer, Onti and Vlachos \cite{dajczer} proved that Einstein submanifolds of space forms with flat normal bundle are locally holonomic.


The study of the intrinsic geometry of hypersurfaces in $\Hi^n \times \R$ and $\Sf^n \times \R$ has been drawn much attention in recent years \cite{aledo1, aledo2, Veken2, Veken3, manfio, novais, Veken1} . Particularly, hypersurfaces with constant sectional curvature were considered by  Aledo, Espinar and Galvez \cite{aledo1, aledo2}, for the two-dimensional case and by Manfio and Tojeiro, \cite{manfio}, for higher dimensions. When $n\geq 4$, Manfio and Tojeiro have proved that a hypersurface with constant sectional curvature $c$ only exists when $c \geq \varepsilon$ and it must be an open part of a complete rotation hypersurface. When $n=3$, $c \in (0,1)$ if $\varepsilon=1$ and $c \in (-1,0)$ if $\varepsilon=-1$. In this case, the hypersurface is constructed explicitly using parallel surfaces in $Q^3(\varepsilon)$. Consequently, the results given by Manfio and Tojeiro in \cite{manfio} and Theorem \ref{thm-einstein} completely solve the problem of the classification of Einstein hypersurfaces in $\Q \times \R$.

\section{Preliminary notions and results}

In this section we will present some preliminary notions and results that will be used in the proof of Theorem \ref{thm-einstein}. Let us first establish some notation. As said before, we will denote by $Q^n(\varepsilon)$ the unit sphere $\Sf^n$, if $\varepsilon=1$, or the hyperbolic space $\Hi^n$ if $\varepsilon=-1$. The Riemannian manifold  $Q^{n}(\varepsilon)\times\mathbb{R}$ will be given in the following models:
$$
\begin{array}{rcl}
        {\Sf}^n \times \R &=& \left\{(x_1,\ldots,x_{n+2})\in\E^{n+2}|\;x_1^2+x_2^2+\ldots+x_{n+1}^2=1\right\}, \\
        {\Hi}^n \times \R &=& \left\{(x_1,\ldots,x_{n+2})\in\Lo^{n+2}|-x_1^2+x_2^2+\ldots+x_{n+1}^2=-1, x_1>0\right\},
\end{array}
$$
with the metric induced by the ambient space. Here $\E^{n+2}$ is the $(n+2)-$dimensional Euclidean space and $\Lo^{n+2}$ is the $(n+2)-$dimensional Lorentzian space with the canonical metric $ds^2=-dx_1^2+dx_2^2+ \ldots +dx_{n+2}^2$.

Let $f : M^n \rightarrow Q^{n}(\varepsilon)\times\mathbb{R}$ be a hypersurface. Denote by $N$ its unit normal and let $\partial_{x_{n+2}}$ be the coordinate vector field of the factor $\R$. Also, let us denote by $T$ the orthogonal projection of  $\partial_{x_{n+2}}$ onto the tangent space of  $M^n$. With this notation, we have the following decomposition
\begin{equation}
\partial_{x_{n+2}}=T+\nu N, \label{decomposition}
\end{equation}
where $\nu$ is a smooth function defined in $M^n$, called \emph{angle function}. 
Let $\nabla$ and $R$ be the Riemannian connection and the curvature tensor of  a hypersurface $f : M^n \rightarrow Q^n(\varepsilon) \times \R$, respectively. It will be considered the following sign convention: $R(X,Y)Z = \nabla_X \nabla_Y Z - \nabla_Y \nabla_X Z - \nabla_{[X,Y]} Z$.  If we denote by $S$ its shape operator, the Gauss equation is given by
    \begin{multline}\label{gauss}
        \langle R(X,Y)Z,W\rangle=\varepsilon(\langle X,W\rangle\langle Y,Z\rangle-\langle X,Z\rangle\langle Y,W\rangle\\
       + \langle X,T\rangle\langle Z,T\rangle\langle Y,W\rangle+\langle Y,T\rangle\langle W,T\rangle\langle X,Z\rangle\\
        -\langle Y,T\rangle\langle Z,T\rangle\langle X,W\rangle-\langle X,T\rangle\langle W,T\rangle\langle Y,Z\rangle)\\
        +\langle SX,W\rangle\langle SY,Z\rangle-\langle SX,Z\rangle\langle SY,W\rangle.
    \end{multline}
Moreover, since the vector field $\partial_{x_{n+2}}$ is parallel in $\Q\times\R$, we have
    \begin{equation}\label{Xcos}
    \begin{array}{rcl}
        \nabla_{X}T&=&\nu SX, \\
           X[\nu]&=&-\langle X,ST\rangle.
     \end{array}    
    \end{equation}
        
At this point, we present a fundamental result that will be used in the proof of Theorem 1. In \cite{tojeiro}, Tojeiro presented a characterization of the hypersurfaces for which $T$ is principal direction. Such characterization is given as follows. 

Let $g: \overline{M}^{n-1} \rightarrow Q^n(\varepsilon)$  be a hypersurface and let $g_s : \overline{M}^{n-1} \rightarrow Q^n (\varepsilon)$, $s \in I \subset \mathbb{R}$, be its family of parallel hypersurfaces, given by
\begin{equation}
g_s(x) = C_{\varepsilon} (s) g(x) + S_{\varepsilon} (s) N(x), \label{parallel-gs}
\end{equation}
where $x \in \overline{M}^{n-1}$, $N$ is a unit normal vector field to $g$ and the functions $C_{\varepsilon}$ and $S_{\varepsilon}$ are given by
\begin{equation}
C_{\varepsilon}(s) = \left\{ 
\begin{array}{l}
\cos(s), \, \textnormal{ if }\, \varepsilon = 1, \\
\cosh(s), \, \textnormal{ if }\, \varepsilon = -1,
\end{array}
\right. \,\,\, \textnormal{ and } \,\,\, 
S_{\varepsilon}(s) = \left\{ 
\begin{array}{l}
\sin(s), \, \textnormal{ if }\, \varepsilon = 1, \\
\sinh(s), \, \textnormal{ if }\, \varepsilon = -1.
\end{array}
\right. \label{e-functions}
\end{equation}
Let $f: M^n := \overline{M}^{n-1} \times I \rightarrow Q^n(\varepsilon) \times \mathbb{R}$ be a hypersurface defined by
\begin{equation}
f(x,s) = g_s(x) + a(s) \partial_{n+2}, \label{function-tojeiro}
\end{equation}
for a smooth function $a : I \rightarrow \mathbb{R}$ with positive derivative. In this context, the following theorem provides the mentioned characterization:

\begin{theorem}[\cite{tojeiro}]
Let $f$ be the map given in (\ref{function-tojeiro}), where $g_s$ is defined by (\ref{parallel-gs}). Then the map $f$ defines, at regular points, a hypersurface that has $T$ as
a principal direction. Conversely, any hypersurface $f: M^n \rightarrow Q^n(\varepsilon) \times \mathbb{R}$, $n \geq 2$, with nowhere vanishing angle function that has $T$ as a principal direction is locally given in this way.
\label{thm-tojeiro}
\end{theorem}

\begin{remark}
The hypersurfaces with the property of the vector field $T$ being a principal direction constitute an important class of hypersurfaces of $Q^n(\varepsilon) \times \R$. This class of hypersurfaces includes the rotation hypersurfaces \cite{Veken3}, the hypersurfaces with constant sectional curvature  \cite{manfio} and the hypersurfaces whose normal direction makes a constant angle with the vector field $\partial_{x_{n+2}}$  \cite{dillen2, dillen3, manfio, tojeiro}. Besides that, it was proved in \cite{tojeiro} that such a property is equivalent to $M^n$ has flat normal bundle as a submanifold into $\E^{n+2}$, resp. $\Lo^{n+2}$. This fact was also obtained for the two-dimensional case in  \cite{Veken4, dillen}, where surfaces of $Q^2(\varepsilon) \times \R$ having $T$ as a principal direction were considered. 
\end{remark}

For a hypersurface given locally by (\ref{function-tojeiro}), one has:
\begin{eqnarray}
|T| = \dfrac{a'(s)}{\sqrt{1+a'(s)^2}}, \label{mt} \\
\nu = \dfrac{1}{\sqrt{1+(a'(s))^2}}. \label{angle-function}
\end{eqnarray}
Also, the principal curvatures are given by
\begin{equation}
\begin{array}{rcl}
\lambda_i &=& - \dfrac{a'(s)}{\sqrt{1+a'(s)^2}} \lambda_i^s, \, 1 \leq i \leq n-1, \\
\lambda_n &=& \dfrac{a''(s)}{(\sqrt{1+a'(s)^2})^3},
\end{array} \label{principal-curvatures}
\end{equation}
where $\lambda_n$ is the principal curvature associated to $T$ and $\lambda_i^s, \, 1 \leq i \leq n-1,$  are the principal curvatures of $g_s$, i.e., 
\begin{equation}
\lambda_i^s = \dfrac{\varepsilon S_{\varepsilon}(s) + \lambda_i^g C_{\varepsilon}(s)}{C_{\varepsilon}(s) - \lambda_i^g S_\varepsilon(s)} \label{principal-curvature-s}.
\end{equation}
where $\lambda_i^g,\, 1 \leq i \leq n-1,$ are the principal curvatures of $g$. Finally, let us observe that, by equations (\ref{mt}) and  (\ref{principal-curvatures}) we have
\begin{equation}
\lambda_n = \dfrac{d|T|}{ds}. \label{kn-edo}
\end{equation}

We also present in this section two results regarding isoparametric hypersurfaces in space forms. We may suggest to the reader as references the survey \cite{cecil} or Section 3.1 in \cite{cecil-book}.  Let us recall that $g: \overline{M}^{n-1} \rightarrow \Q$ is said to be an isoparametric hypersurface if it has constant principal curvatures. In \cite{cartan-iso}, Cartan proved that a hypersurface $g: \overline{M}^{n-1} \rightarrow \Q$ is isoparametric if and only if each parallel hypersurface $g_s$ as given in \eqref{parallel-gs} has constant mean curvature, i.e., the mean curvature of $g_s$ depends only on $s$ (see Theorem 3.6 in \cite{cecil-book}). In the same paper, Cartan established an important relation between the principal curvatures of isoparametric hypersurfaces. This relation is known as \emph{Cartan's identity} (or \emph{Cartan's formula}, following \cite{cecil-book}, page 91) and it is given as follows: let $g: \overline{M}^{n-1} \rightarrow \Q$ be an isoparametric hypersurface with $d$ distinct principal curvatures and respective multiplicities $m_1, \ldots, m_d$. If $d>1$, for each $i$, $1 \leq i \leq d$ one has
\begin{equation}
\displaystyle \sum_{j \neq i} m_j \dfrac{\varepsilon+\lambda_i \lambda_j}{\lambda_i - \lambda_j} = 0. \label{cartan-id}
\end{equation}

In order to prove Theorem 1, we will need the following lemmas. The first will establish the Ricci tensor on a hypersurface $f: M^{n} \rightarrow \Q \times \mathbb{R}$ while the second will show that, on an Einstein hypersurface, the vector field $T$ is an eigenvector of the shape operator at $p \in M$, as long as $T \neq 0$ at $p$. 

In what follows, the Ricci tensor is given by  
\begin{equation}
\textnormal{Ric}(Y,Z)=\textnormal{trace}\left\{ X \mapsto R(X,Y)Z \right\}. \label{ricci-general}
\end{equation}
\begin{lemma}
Let $M^n$ be a hypersurface in $\Q \times \R$, then the Ricci tensor of $M^n$ is given by
\begin{equation}
\begin{array}{rcl}
\textnormal{Ric}(Y,Z) &=& \varepsilon(n-1-|T|^2) \langle Y, Z \rangle + \varepsilon (2-n) \langle Y, T \rangle \langle Z, T \rangle \\
&& + n H \langle SY, Z \rangle - \langle SY, SZ \rangle, 
\end{array}
\label{ricci-tensor}
\end{equation}
where $Y, \, Z$ are arbitrary vector fields on $M^n$ and $H$ is the mean curvature.
\end{lemma}

\begin{proof}
Let $\left\{ e_i \right\}_{i=1}^n$ an orthonormal basis of principal directions, with $Se_i = \lambda_i e_i$. If we write $Y = \displaystyle \sum_{k=1}^n y_k e_k$, $Z = \displaystyle \sum_{k=1}^n z_k e_k$ and $T = \displaystyle \sum_{k=1}^n t_k e_k	$, it follows by Gauss Equation (\ref{gauss}) that 
\begin{equation}
\begin{array}{rcl}
\langle R(e_k, Y)Z, e_k \rangle &=& \varepsilon \left[ \langle Y, Z \rangle - y_k z_k + t_k y_k \langle Z, T \rangle + t_k z_k \langle Y, T \rangle - \right. \\
&& \left.  - \langle Y, T \rangle \langle Z, T \rangle - t_k^2 \langle Y, Z \rangle  \right] + \\
&& + \lambda_k \langle SY, Z \rangle - \langle e_k, SZ \rangle \langle SY, e_k \rangle. \label{gauss-equation-ricci}
\end{array}
\end{equation}
Consequently, by \eqref{gauss-equation-ricci} and \eqref{ricci-general}, the Ricci tensor is given by \eqref{ricci-tensor}.
\end{proof}

The next lemma give a characterization of Einstein hypersufaces with $T \neq 0$.

\begin{lemma}
Let $M^n$, $n > 3$, be an Einstein hypersurface in $Q^n(\varepsilon) \times \mathbb{R}$. If $T \neq 0$ at $p \in M^n$, then $T$ is an eigenvector for the shape operator at $p$. \label{eigen-t}
\end{lemma}
\begin{proof}
Let $\left\{ e_i \right\}_{i=1}^n$ an orthonormal basis of principal directions, with $Se_i = \lambda_i e_i$. Let us write $T = \displaystyle \sum_{k=1}^n t_k e_k$. If $T \neq 0$ at $p \in M^n$, there is at least one coefficient $t_k \neq 0$. Since $M^n$ is an Einstein manifold, its Ricci tensor satisfy
$$
\textnormal{Ric}(e_i,e_j) = \rho \delta_{ij},
$$
for some constant $\rho$. When we consider the Ricci tensor applied on the orthonormal basis $\left\{ e_i \right\}_{i=1}^n$, we have
\begin{equation}
\textnormal{Ric}(e_i,e_j) = \left[ \varepsilon(n-1-|T|^2) + n H \lambda_i - \lambda_i \lambda_j \right] \delta_{ij} + \varepsilon (2-n) t_i t_j \label{ricci-basis}.
\end{equation}
By Equation (\ref{ricci-basis}) we must have
\begin{equation}
\left[ \varepsilon(n-1-|T|^2) + n H \lambda_i - \lambda_i \lambda_j - \rho \right] \delta_{ij} + \varepsilon (2-n) t_i t_j = 0 \label{einstein-condition}
\end{equation}
and we conclude that $t_i t_j = 0$, for all $i, \, j$, with $i \neq j$.  Consequently, there is only one coefficient $t_k \neq 0$ and therefore $T = t_k e_k$ at $p$.
\end{proof}

\section{Proof of the main result}

\begin{proof}[Proof of Theorem \ref{thm-einstein}] If $T \equiv 0$, then $M^n$ is an open part of a  slice $\Q \times \left\{ t_0 \right\}$, where $t_0 \in \R$. Since the slices are isometric to $\Q$, $M^n$ is a manifold with constant sectional curvature $\varepsilon$. Otherwise, let $\Omega$ be the open, non-empty subset where $|T|> 0$. By Lemma \ref{eigen-t}, $T$ is a principal direction in $\Omega$. Without loss of generality we can write $T = t_n e_n$ and $ST = \lambda_n T$. Since $M^n$ is Einstein, we have from Equation (\ref{einstein-condition}) that
\begin{eqnarray}
 \varepsilon(n-1-|T|^2) + nH \lambda_i - \lambda_i^2 - \rho &=& 0, \, \textnormal{ for } 1 \leq i \leq n-1 \label{equation-lambda-i} \\
\varepsilon(n-1)(1-|T|^2) + nH \lambda_n - \lambda_n^2 - \rho &=& 0. \label{equation-lambda-n} 
\end{eqnarray}

Equation \eqref{equation-lambda-i} implies that we have at most two distinct principal curvatures among the $(n-1)$ first principal curvatures. In fact, from \eqref{equation-lambda-i} we have
\begin{equation}
(\lambda_i - \lambda_j)(nH - \lambda_i - \lambda_j) = 0 , \, \textnormal{ for } 1 \leq i, \, j \leq n-1. \label{lambda-i-j}
\end{equation}

Let us suppose by contradiction that there are three distinct principal curvatures $\lambda_{i_1}, \, \lambda_{i_2}, \, \lambda_{i_3}$. It follows by equation of (\ref{lambda-i-j}) that
$$
\begin{array}{rcl}
\lambda_{i_1} + \lambda_{i_2} &=& n H, \\
\lambda_{i_2} + \lambda_{i_3} &=& n H, \\
\lambda_{i_3} + \lambda_{i_1} &=& n H.
\end{array}
$$
The equations above implies that $\lambda_{i_1} =  \lambda_{i_2} = \lambda_{i_3}$, which is a contradiction.

If $\lambda_1 = \lambda_2 = \ldots = \lambda_{n-1}=\mu$, the sectional curvature is constant. In fact, by equation \eqref{gauss-equation-ricci} we have
\begin{eqnarray}
\langle R(e_i,e_j)e_j,e_i\rangle &=& \varepsilon+\mu^2, \,\,\, 1 \leq i, \, j \leq n-1 \label{sectional-ij} \\ 
\langle R(e_i,e_n)e_n,e_i\rangle &=& \varepsilon(1-|T|^2) + \mu \lambda_n.  \label{sectional-in}
\end{eqnarray}
It follows from \eqref{equation-lambda-n} that 
\begin{equation}
\varepsilon(1-|T|^2) + \mu \lambda_n = \dfrac{\rho}{n-1}, \label{constant-sectional-n}
\end{equation}
therefore equation \eqref{sectional-in} implies that $\langle R(e_i,e_n)e_n,e_i\rangle$ is constant. 
By equation (\ref{equation-lambda-i}) we have
\begin{equation}
\varepsilon(1-|T|^2) + \mu \lambda_n = \rho - (n-2)(\mu^2 + \varepsilon). \label{constant-sectional-i}
\end{equation}
When we combine equations (\ref{constant-sectional-n}) and (\ref{constant-sectional-i}), we have from \eqref{sectional-ij} that 
$$\langle R(e_i,e_j)e_j,e_i\rangle = \dfrac{\rho}{n-1}$$
and, consequently, the sectional curvature is equal to $\dfrac{\rho}{n-1}$ in $\Omega$.

Next we will show that the possibility of two distinct principal curvatures does not occur. In this case, we can consider as the two distinct principal curvatures $\lambda_1$ and $\lambda_2$ and therefore there are $p$ principal curvatures equal to $\lambda_1$ and $q$ principal curvatures equal to $\lambda_2$, with $\lambda_1 \neq \lambda_2$ and $p+q=n-1$. By equation \eqref{lambda-i-j} we have $\lambda_1 + \lambda_2 = nH$, consequently,
\begin{eqnarray}
\lambda_n &=& (1-p) \lambda_1 + (1-q) \lambda_2,  \label{first-edo} \\ 
\lambda_1 \lambda_2 &=& \rho - \varepsilon(n-1-|T|^2). \label{second-edo}
\end{eqnarray}
where (\ref{second-edo}) is obtained when we substitute $\lambda_1 + \lambda_2 = nH$ into (\ref{equation-lambda-i}).

We will show that $\lambda_n \equiv 0$ in $\Omega$ and this fact will lead us to a contradiction. If $\nu \equiv 0$, it follows by \eqref{kn-edo} that $\lambda_n \equiv 0$ in $\Omega$. Otherwise, there is a point $p_0$ where $\nu(p_0) \neq 0$ and an open neighborhood $\Omega_0 \subset \Omega$ of $p_0$ such that $\nu \neq 0$. Therefore, by Lemma \ref{eigen-t}, we can apply Theorem \ref{thm-tojeiro} to conclude that $\Omega_0$ is given locally by (\ref{function-tojeiro}). In this case, \eqref{equation-lambda-n} implies that
\begin{equation}
\begin{array}{rcl}
\varepsilon(n-1)(1-|T|^2) + \lambda_n (n-1)H_{g_s}  - \rho &=& 0,
\end{array}
\label{einstein-tojeiro}
\end{equation}
where $H_{g_s}$ is the mean curvature of the parallel $g_s$. 

Let us suppose by contradiction $\lambda_n \neq 0$ in $\Omega_0$. It follows by (\ref{einstein-tojeiro}) that $H_{g_s}$ depends only on $s$, once that equations (\ref{mt})  and  (\ref{principal-curvature-s}) imply that the functions $|T|^2$ and $\lambda_n$ depend only on $s$. In this case, the mean curvature of the parallel $g_s$ depends only on $s$ which implies that $g$ is an isoparametric hypersurface, with two distinct principal curvatures. By Cartan's identity \eqref{cartan-id} we must have 
\begin{equation}
    \lambda_1^g \lambda_2^g + \varepsilon = 0. \label{cartan-1-2}
\end{equation}
 
In this case, it follows directly from \eqref{mt}, (\ref{principal-curvatures}), (\ref{principal-curvature-s}) and \eqref{cartan-1-2} that 
\begin{equation}
\lambda_1 \lambda_2 = - \varepsilon |T|^2. \label{product-curvatures-einstein}
\end{equation}
When we replace (\ref{product-curvatures-einstein}) in (\ref{second-edo}) we have 
$$|T|^2 = \dfrac{\varepsilon(n-1)-\rho}{2\varepsilon},$$
which implies that $|T|^2$ is constant. By equation (\ref{kn-edo}), it follows that $\lambda_n = 0$ in $\Omega_0$, which is a contradiction.

Therefore, we have $\lambda_n \equiv 0$ in $\Omega$. It follows by (\ref{einstein-tojeiro}) that 
\begin{equation}
|T|^2 = \dfrac{\varepsilon(n-1)-\rho}{\varepsilon(n-1)}. \label{mt-constant-einstein}
\end{equation}
In this case, equations (\ref{first-edo}) and (\ref{second-edo}) are rewritten as
\begin{eqnarray}
(p-1) \lambda_1 + (q-1) \lambda_2&=&0,  \label{first-edo-2-einstein} \\ 
\lambda_1 \lambda_2 &=& \left( \dfrac{n-2}{n-1} \right) (\rho - \varepsilon(n-1)). \label{second-edo-2-einstein}
\end{eqnarray}

Since $|T| \neq 0$, equations \eqref{mt}, (\ref{principal-curvatures}), \eqref{mt-constant-einstein}, \eqref{first-edo-2-einstein} and \eqref{second-edo-2-einstein} imply that
\begin{eqnarray}
(p-1) \lambda_1^s + (q-1) \lambda_2^s&=&0,  \label{first-edo-3-einstein} \\ 
\lambda_1^s \lambda_2^s &=& -\varepsilon(n-2). \label{second-edo-3-einstein}
\end{eqnarray}

We claim that the system above has no solution for $n>3$. In fact, Equations \eqref{first-edo-3-einstein} and \eqref{second-edo-3-einstein} imply that $\lambda_i^s$ are constants, unless $p=q=1$, which is not the case since $p+q=n-1$. Therefore, evaluating in $s=0$ we conclude that $g$ is isoparametric. By Cartan's identity \eqref{cartan-id}, $\lambda_1^g \lambda_2^g + \varepsilon = 0$. This fact with Equation \eqref{second-edo-3-einstein} in $s=0$ implies $n=3$, which is not the case. 

We conclude that, in the open subset $\Omega$ where $|T| > 0$, the sectional curvature is a constant $K_0 = \dfrac{\rho}{n-1}$. If $M^n \setminus \Omega$ has empty interior, we have by continuity that $M^n$ has constant sectional curvature $K_0$. Otherwise, there is an open subset $\mathcal{O} \subset M^n \setminus \Omega$, where $T \equiv 0$. As we saw at the beginning of the proof,  $\mathcal{O}$ is an open part of a  slice $\Q \times \left\{ t_1 \right\}$, for some $t_1 \in \R$, and the sectional curvature in $\mathcal{O}$ is constant equal to $\varepsilon$, which implies $\rho = (n-1) \varepsilon$. Since $\rho$ is constant in $M^n$, we must have $K_0 = \varepsilon$ and the sectional curvature in $\Omega \cup \mathcal{O}$ is $\varepsilon$, for any open subset $\mathcal{O}$ where $T \equiv 0$. Again we use the continuity of the sectional curvature to conclude that $M^n$ has constant sectional curvature equal to $\varepsilon$.
\end{proof}

\bibliographystyle{amsplain}

\end{document}